\documentclass[twoside]{article}
\usepackage{dsfont,amsmath,amsthm,amssymb,mathrsfs,fancyhdr,marvosym,wasysym}
 \newtheorem{thm}{Theorem}[section]
 \newtheorem{cor}[thm]{Corollary}
 \newtheorem{lem}[thm]{Lemma}
 
 \theoremstyle{definition}
 \newtheorem{defn}[thm]{Definition}
 \theoremstyle{remark}
 \newtheorem{rem}[thm]{Remark}
 \newtheorem*{ex}{Example}
 \numberwithin{equation}{section}

\title{Carnot-Carath\'eodory Metrics in \\ Unbounded Subdomains of $\mathbb{C}^2$}
\author{Aaron Peterson\footnote{This work was supported in part by NSF Grant No. 1147523 - RTG: Analysis and Applications at the University of Wisconsin - Madison. The author would like to thank Professor Alexander Nagel for his support throughout this project, and Professor Brian Street for helpful conversations.}}
\date{}
\begin{document}

\maketitle

\begin{abstract}
We introduce a new class of unbounded model subdomains of $\mathbb{C}^2$ for the $\Box_b$ problem. 
Unlike previous finite type models, these domains need not be bounded by algebraic varieties. 
In this paper we obtain precise global estimates for the Carnot-Carath\'eodory metric induced on the boundary of such domains by the real and imaginary parts of the CR vector field.
\end{abstract}

\section{Introduction}
Let $\Omega=\left\{(z_1,z_2)\ :\ \Im(z_2)>P(z_1)\right\}\subset \mathbb{C}^2$, where $P:\mathbb{C}\rightarrow \mathbb{R}$ is smooth, subharmonic, and non-harmonic. 
One fundamental problem for such domains, solved when $P$ is a polynomial by Nagel, Stein, and Wainger in \cite{NagelSteinWainger1985}, is to describe the Carnot-Carath\'eodory (CC) metric $d(p,q)$ induced on ${\rm b}\Omega$ by the real and imaginary parts of the tangential CR vector field.
In two dimensions this metric controls several analytical objects associated to $\Omega$, such as the Bergman kernel, the Szeg\H{o} kernel, and parametrices for the $\Box$ and $\Box_b$ operators; see \cite{NagelStein2006} and the references therein for examples.
When $P$ is a polynomial, these objects have been studied via scaling arguments; see in particular \cite{NagelRosaySteinWainger1989} for an example.
We are interested in studying these objects, and therefore the metric $d(p,q)$, in domains where the function $P$ is not a polynomial.

Following the notation of \cite{NagelSteinWainger1985}, we denote by $\Lambda(p,\delta)$ the diameter (in the $\Re(z_2)$-direction) of the CC ball on ${\rm b}\Omega$ with center $p=(p_1,p_2)$ and radius $\delta$.
The purpose of this paper is to give an elementary alternate description of the balls $B_d(p,\delta)$ associated to $d$ which does not require homogeneity of $P$.
This necessitates a description of $\Lambda(p,\delta)$, which we accomplish by introducing the notion of a \emph{$(p_1,\delta)$-stockyard}, a collection $R$ of subsets of $\mathbb{C}$ which satisfies some mild connectedness and regularity properties.

Our first main result is the following.

\begin{thm}\label{gschar}$\displaystyle\Lambda(p,\delta) = \genfrac{}{}{0pt}{1}{\sup}{(p_1,\delta)-stockyards\ R}\ \sum_{R_i\in R} \int_{R_i} \Delta P(w) dm(w).$
\end{thm}
\smallskip

We next introduce a class of smooth non-polynomial functions $P$ for which $\Lambda(p,\delta)\approx f(\delta)$ when $\delta\geq \delta_0$ is sufficiently large.
If such a function $f(\delta)$ exists, then $(f(\delta),\delta_0)$ is said to be a \emph{uniform global structure} (UGS) for ${\rm b}\Omega$.
The presence of a UGS imposes strong conditions on $\Delta P$, which in turn restricts the possibilities for $f(\delta)$.
In particular, we have our second main result.

\begin{thm}\label{ugschar} If ${\rm b}\Omega$ has a UGS $(f(\delta),\delta_0)$, then $\delta\lesssim f(\delta)\lesssim \delta^2$.
\end{thm}
\smallskip

Using these estimates, we are able to parametrize large CC balls and obtain explicit formulas for $d(p,q)$ when $p$ and $q$ are far apart.
These estimates will be extensively used in future works, where we will study the Szeg\H{o} kernel and related objects on these non-polynomial model domains.

The rest of the paper is structured as follows. 
Section \ref{sec:notation} sets the notation that will be used throughout the paper. 
In Section \ref{sec:globalstructure} we prove Theorem \ref{gschar} and provide two extreme examples of uniform global structures. 
Section \ref{sec:uniformglobalstructures} characterizes the possible uniform global structures, and contains a proof of Theorem \ref{ugschar}. 
We parametrize the balls $B_d(p,\delta)$ and give large-scale estimates for the metric $d$ in Section \ref{sec:metric}.
Finally, Section \ref{sec:conclusion} concludes the paper.

\pagestyle{fancy}
\renewcommand{\leftmark}{}
\renewcommand{\rightmark}{}

\section{Notation}\label{sec:notation}

We write $A\lesssim B$ when there exists a constant $C$, independent of all relevant quantities, such that $A\leq CB$. 
We will say $A\gtrsim B$ if $B\lesssim A$, and we write $A\approx B$ if both $A\lesssim B$ and $B\lesssim A$. 

When working in $\mathbb{C}\approx \mathbb{R}^2$, $dm(\cdot)$ will denote Lebesgue measure with respect to $\cdot$ and $|z|$ will denote the modulus of $z$. We will use the notation $B(z,\delta)$ (no subscript) to denote $\left\{ w\ :\ |z-w|<\delta\right\}$.

To simplify some of the proofs, we define a special class of real functions on $[0,1]$. 
We say that $f\in FPWS[0,1]$ if there exist a finite number of points $0=a_0<a_1<\ldots<a_N<a_{N+1}=1$ such that, for all $i=0,\dots,N$, $f$ is smooth on $(a_i,a_{i+1})$ and $f|_{(a_i,a_{i+1})}$ extends continuously to $[a_i,a_{i+1}]$.
If $A\subset\mathbb{R}^2$, then ${\rm b}A$ is FPWS if it can be locally parametrized by a FPWS function.

\section{The Global Structure}\label{sec:globalstructure}

Let $\Omega$ be as in the introduction.
Identify ${\rm b}\Omega$ with $\mathbb{C}\times\mathbb{R}\approx \mathbb{R}^3$ via $(z_1,t+iP(z_1))\sim (z_1,t)\sim (x,y,t)$, and equip ${\rm b}\Omega$ with Lebesgue measure. 
The space of tangential Cauchy-Riemann operators on ${\rm b}\Omega$ is spanned by $\bar{Z}=\partial_{\bar{z}_1}-2i\partial_{\bar{z}_1}P(z_1)\partial_{\bar{z}_2},$ which under the identification with $\mathbb{C}\times\mathbb{R}$ becomes $\bar{Z}=\partial_{\bar{z}_1}-i\partial_{\bar{z}_1}P(z_1)\partial_{t}.$ 
Decompose $\bar{Z}=X-iY$, where $X$ and $Y$ are the real vector fields given by \[X=\frac{1}{2}\partial_{x}+\frac{1}{2}\partial_{y}P(x,y)\partial_{t},\qquad Y=-\frac{1}{2}\partial_{y}+\frac{1}{2}\partial_{x}P(x,y)\partial_{t}.\]

Let $d:{\rm b}\Omega\times {\rm b}\Omega\rightarrow [0,\infty)$ be the Carnot-Carath\'eodory metric associated to the vector fields $X$ and $Y$. 
That is, 
\begin{eqnarray*} d(p_0,p_1)&=& \inf\left\{ \delta\ :\ \exists \gamma:[0,1]\rightarrow {\rm b}\Omega,\ \gamma(0)=p_0,\ \gamma(1)=p_1,\right.\\ 
														& & \qquad\qquad \gamma'(t)=\alpha(t)\delta X(\gamma(t))+\beta(t)\delta Y(\gamma(t))\ a.e.,\\
														& & \qquad\qquad \left. \alpha,\beta\in FPWS[0,1],\ |\alpha(t)|^2+|\beta(t)|^2<1\ a.e.\right\}.
\end{eqnarray*}

We will use the notation $B_d(p,\delta):=\left\{ q\ :\ d(p,q)<\delta\right\}$ to denote the CC-ball with center $p$ and radius $\delta$. 

Let $\mathscr{X}=(FPWS[0,1])^2\cap \left\{ (\alpha,\beta)\ :\ |\alpha(s)|^2+|\beta(s)|^2<1\right\}$, and \[\mathscr{X}^\ast=\left\{ (\alpha,\beta)\in\mathscr{X}\ :\ \int_0^1 \alpha(s)ds=0,\ \int_0^1 \beta(s)ds=0 \right\}.\]
For $(\alpha,\beta)\in\mathscr{X}$, $p_0=(z_0,t_0)=(x_0,y_0,t_0)\in {\rm b}\Omega$, and $\delta>0$, we define $F_{(\alpha,\beta),\delta}(p_0)=\gamma(1)$, where $\gamma:[0,1]\rightarrow {\rm b}\Omega$ is the path with \begin{equation}\label{ccpath} \gamma(0)=p_0\qquad\mbox{and}\qquad \gamma'(s)=\alpha(s)\delta X(\gamma(s))+\beta(s)\delta Y(\gamma(s)),\ a.e.\end{equation}

By setting $a=\int_0^1 \alpha(t)dt$ and $b=\int_0^1 \beta(t)dt$, we get \[F_{(\alpha,\beta),\delta}(p_0)=(x_0+\frac{\delta a}{2},y_0-\frac{\delta b}{2},t_0+\Lambda_{z_0,\delta}(\alpha,\beta)),\] where \begin{equation}\label{deflambda}\Lambda_{z_0,\delta}(\alpha,\beta)=\int_0^1 \left[ \partial_{y}P(\gamma(r))\gamma_1'(r)-\partial_{x}P(\gamma(r))\gamma_2'(r)\right]dr,\end{equation} \[\gamma_1(r)=x_0+\frac{\delta}{2}\int_0^r \alpha(s)ds,\qquad \gamma_2(r)=y_0-\frac{\delta}{2}\int_0^r \beta(s)ds.\]

Let $\pi\gamma$ be the projection of $\gamma$ onto the $xy$-plane. 
We denote by $L(\pi\gamma):=\delta\int_0^1 \sqrt{|\alpha(t)|^2+|\beta(t)|^2}dt$ the length of $\pi\gamma$. 
If $\alpha$ and $\beta$ have mean $0$ (i.e. if $(\alpha,\beta)\in\mathscr{X}^\ast$), then the projection $\pi\gamma$ is a closed curve in $\mathbb{C}$. 
Analyzing the integral (\ref{deflambda}) over this curve is essential to understanding the structure of $d$. 
With this view we make the following definition.

\begin{defn}{\rm Define $\Lambda(p_0,\delta):= \displaystyle\sup_{(\alpha,\beta)\in\mathscr{X}^\ast} \left| \Lambda_{z_0,\delta}(\alpha,\beta)\right|.$
}\end{defn}

\begin{rem}\label{revpaths}
{\rm $\Lambda(p_0,\delta)$ measures the distance, in the $t$ direction, from $p_0$ to the boundary of $B_d(p_0,\delta)$. 
Since the coefficients of the ODE (\ref{ccpath}) are independent of $t$, we see that the distance $\Lambda(p_0,\delta)$ is actually independent of $t_0$.
}\end{rem}

If $\pi\gamma$ above is a simple, piecewise smooth, negatively oriented closed curve bounding a region $R$, then an application of Green's theorem gives \[\Lambda_{z_0,\delta}(\alpha,\beta)=\int_R \Delta P(w) dm(w),\] which states that the amount that $\gamma$ alters the $t$-coordinate of our initial point is equal to the mass of $\Delta P$ contained within $R$.

In general, though, $\pi\gamma$ may be far from simple. 
To work around this, we will replace $\pi\gamma$ by a collection of simple, piecewise smooth, negatively oriented closed curves by essentially breaking the trace of $\gamma$ into the traces of such curves. 
We describe $\Lambda_{z_0,\delta}(\alpha,\beta)$ by studying the collection of regions bounded by these curves. 
To this end, we make the following definition.

\begin{defn}\label{stockyarddef}{\rm 
We say $A\subset\mathbb{C}$ is a \emph{pen} if $A$ is open, connected, simply connected, and if ${\rm b} A$ is $FPWS$. 
We call $P(A)=L({\rm b}A)$ (the perimeter of $A$) the amount of \emph{fencing} used to enclose $A$. 
For a fixed $z_0\in \mathbb{C}$ and $\delta>0$, we say that a finite collection of pens $R=\left\{ R_1,\ldots,R_N\right\}$ is a \emph{$(z_0,\delta)$-stockyard} if 
\[z_0\in \bigcup_{i=1}^N {\rm b} R_i,\quad \displaystyle\sum_{i=1}^N {\rm P}({\rm b} R_i)\leq\delta,\quad\mbox{and}\quad \bigcup_{i=1}^N{\rm b} R_i\ \mbox{is connected}.\]
A $(z_0,\delta)$-stockyard is therefore comprised of a chain of nonempty pens, where the union of the boundaries of the pens is continuous, and the total amount of fencing used to enclose the pens is bounded by $\delta$. 
}\end{defn}

We are now in a position to prove Theorem \ref{gschar}.

\begin{proof}[Proof of Theorem \ref{gschar}]
Because $\Lambda_{z_0,\delta}(\alpha,\beta)$ can be viewed as a line integral in the plane, for $(\alpha,\beta)\in\mathscr{X}$ we will say `the curve generated by $(\alpha,\beta)$' to refer to the curve $\gamma:[0,1]\rightarrow\mathbb{C}$ such that $\gamma(0)=z_0$ and, for almost every $t$, $\gamma'(t)= \frac{\alpha(t)\delta}{2}\partial_x\big|_{\gamma(t)}- \frac{\beta(t)\delta}{2}\partial_y\big|_{\gamma(t)}$.

Fix a nontrivial $(\alpha_0,\beta_0)\in\mathscr{X}^\ast$, and let $\gamma_0$ be the curve determined by $(\alpha_0,\beta_0)$. 
The proof of the theorem is now an advanced calculus exercise, which we break into several steps.

First, uniformly approximate $\gamma_0$ by a piecewise linear closed curve $\gamma_1$ (generated by $(\alpha_1,\beta_1)\in\mathscr{X}^\ast$) with $L(\gamma_1)< L(\gamma)$, $\gamma_1(0)=z_0$, and where $\Lambda_{z_0,\delta}(\alpha_1,\beta_1)$ is arbitrarily close to $\Lambda_{z_0,\delta}(\alpha_0,\beta_0)$. 
By slightly perturbing $\gamma_1$, we obtain another piecewise linear closed curve $\gamma_2$ (generated by $(\alpha_2,\beta_2)\in\mathscr{X}^\ast$) with $L(\gamma_2)\leq L(\gamma)$, $\gamma_2(0)=z_0$, and where $\gamma_2$ has finitely many self-intersections.

We now decompose $\gamma_2$ into the union of oriented curves $\eta_1,\ldots,\eta_N$, where $\sum L(\eta_i)\leq \delta$. 
To do this, let $\mathscr{V}=\left\{V_1=z_0,V_2,\ldots,V_N\right\}$ denote the finitely-many self intersection points of $\gamma_2$ (together with the point $z_0$), and let $\left\{ t_i\right\} = (\gamma_2)^{-1}(\mathscr{V})$. 
Without loss of generality, we may assume that $0=t_0<t_1<\cdots<t_{M}=1$. 
For $i=1,\ldots M$, define the edge $E_i$ by $E_i = (\gamma_2(t_{i-1}),\gamma_2(t_i))$. 
Setting $\mathscr{E}=\left\{ E_1,\ldots E_M\right\}$, we see that $(\mathscr{E},\mathscr{V})$ is a directed Eulerian graph.
Since $(\mathscr{E},\mathscr{V})$ is Eulerian, we may decompose it into a finite number of edge-disjoint directed cycles $C_1,\ldots,$ $C_N$. 
Parametrizing $C_i$ by a curve $\eta_i$ whose orientation agrees with the orientation of $C_i$, we see that $\eta_i$ is a (piecewise linear) oriented simple closed curve. 
Moreover, by our construction of $\gamma_2$, $\sum L(\eta_i)\leq\delta$, as desired.

We now show that $\displaystyle{\Lambda}(p_0,\delta) \leq \genfrac{}{}{0pt}{1}{\sup}{(z_0,\delta)-stockyards\ R}\ \sum_{R_i\in R} \int_{R_i} \Delta P(w) dm(w)$. 
Denote by $R_i$ the region bounded by $\eta_i$, and apply Green's theorem to obtain 
$$-\displaystyle\oint_{\gamma_2}\partial_{y}Pdx-\partial_{x}Pdy = \displaystyle\sum_{i=1}^N (-1)^{\mbox{sign of}\ \eta_i}\displaystyle\int_{R_i} \Delta P(w)dm(w).$$
By changing the orientations of the positively oriented $\eta_i$,
$$\displaystyle\oint_{\gamma_2}\partial_{y}Pdx-\partial_{x}Pdy \leq \displaystyle\sum_{i=1}^N \displaystyle\int_{R_i} \Delta P(w)dm(w),$$
which, because $(R_1,\ldots,R_N)$ is a $(z_0,\delta)$-stockyard, implies 
$$\displaystyle{\Lambda}(p_0,\delta) \leq \genfrac{}{}{0pt}{1}{\sup}{(z_0,\delta)-stockyards\ R}\ \sum_{R_i\in R} \int_{R_i} \Delta P(w) dm(w).$$

Let $R=(R_1,\ldots,R_N)$ be a $(z_0,\delta)$-stockyard. 
It remains to show that there exists a FPWS curve $\gamma$ with $L(\gamma)\leq \delta$, $\gamma(0)=z_0$, and $$\displaystyle\oint_{\gamma}\partial_{y}Pdx-\partial_{x}Pdy \leq \displaystyle\sum_{i=1}^N \displaystyle\int_{R_i} \Delta P(w)dm(w).$$ 

If $R=(R_1)$, then we merely parametrize ${\rm b} R_1$ in the negative direction to get the result. 
If we are able to produce such a curve $\tilde{\gamma}$ for $\tilde{R}=(R_1,\ldots,R_{N-1})$, then we choose some $t^\ast\in[0,1]$ such that $\tilde{\gamma}(t^\ast)\in {\rm b} R_N$. 
Reparametrize $\tilde\gamma$ so that $t^\ast=0$. 
If $\eta$ is a negatively oriented parametrization of ${\rm b} R_N$ with $\eta(0)=\tilde\gamma(1)$, then we concatenate these two curves together and reparametrize to get a curve $\gamma$ with the required properties. 
This concludes the proof.		
\end{proof}
\smallskip

\begin{defn}{\rm The function ${\Lambda}:{\rm b}\Omega\times (0,\infty)\rightarrow [0,\infty)$ is called the \emph{global structure} of ${\rm b}\Omega$. If there exists $\delta_0>0$ and a function $f:[\delta_0,\infty)\rightarrow [0,\infty)$ such that $\Lambda(p_0,\delta)\approx f(\delta)$ for $\delta\geq \delta_0$, with constants depending only on $\delta_0$, then we say that $(f(\delta),\delta_0)$ is a \emph{uniform global structure} (UGS) of ${\rm b}\Omega$.
}\end{defn}

\begin{ex}[Example with UGS $(\delta^2,\delta_0)$]\label{quadugs}{\rm If $P(z)=|z|^2$, then $\Delta P \equiv 1.$ A trivial application of the isoperimetric inequality implies that $\Lambda(p_0,\delta)\equiv \frac{1}{\pi}\delta^2$. Hence, ${\rm b}\Omega$ has UGS $(\delta^2,0)$.
}\end{ex}

\begin{ex}[Example with UGS $(\delta,\delta_0)$]\label{linearugs}{\rm 
Enumerate $\left\{ (10m,10n):(m,n)\in\mathbb{Z}^2\right\}$ by $\left\{ c_k\right\}_{k\in\mathbb{N}}$ and let $D_k$ be the disc with center $c_k$ and radius $r_k=2^{-k}$.
Define $h_k=\frac{1}{\pi}2^k$, so that $\pi r_k^2 h_k = 2^{-k}$, and let $H(z)=h_k$ if $z\in D_k$, and $H(z)=0$ otherwise. 
Then $H(z)\geq 0$ and $\int_{\mathbb{C}} H(w)dm(w)=1$. 
Moreover, by convolving with the Newtonian potential, we get a function $P(z)$ with $\Delta P=H$.
Also, for $k\geq 0$, we trivially have $2^k \sup_{z} \int_{B(z,2^{-k})} \Delta P(w)dm(w)=1.$
We show that, for $P$, we have $\Lambda(p_0,\delta)\approx \delta$ if $\delta\geq 30$.

For $z\in \mathbb{C}$ and $\delta\geq 30$, consider the stockyard $R$ given by the following construction. 
Choose $k$ such that ${\rm dist}(z,D_k)=\min_l {\rm dist}(z,D_l)$. 
Let $R_0$ be a pen which touches both $z$ and $z_b\in{\rm b} D_k$, and has essentially minimal perimeter (i.e. $P(R_0)\approx {\rm dist}(z,D_k)$). 
By construction, we have $P(R_0)\leq 20$. 
Of course, if $z\in{\rm b} D_k$ then no $R_0$ is necessary. 
Choose $N=\lfloor (2\pi)^{-1}2^k(\delta-P(R_0))\rfloor$, and let each of $R_1,\ldots,R_N$ be $D_k$. 
Then $(R_0,\ldots,R_N)$ is a $(z,\delta)$-stockyard. 
Moreover, because $0\leq \int_{R_0}\Delta P(w)dm(w) \leq \int_{D_k}\Delta P(w)dm(w)$, we see that $\sum_{i=0}^N \int_{R_i} \Delta P(w)dm(w) \approx N 2^{-k}\approx \delta.$

To see that this is essentially the best that can be done, let $R=(R_1,\ldots,R_N)$ be an arbitrary $(z,\delta)$-stockyard. 
If $R_i\cap D_k=\emptyset$ for all $k$, then $\int_{R_i}\Delta P(w) dm(w)=0$.
If $R_i\cap D_k\neq \emptyset$ and $P(R_i)\leq 1$, then by the isoperimetric inequality we have $$\displaystyle\int_{R_i} \Delta P(w) dm(w) \leq \begin{cases} P(R_i)^2 2^k,\quad &\mbox{if}\ P(R_i)\leq r_k,\\ 2^{-k},\quad&\mbox{if}\ P(R_i)\geq r_k.\end{cases}$$ 
Hence, if $P(R_i)\leq 1$ then $\int_{R_i}\Delta P(w) dm(w) \leq P(R_i).$
Because $\|\Delta P\|_{L^1}=1$, for $P(R_i)\geq 1$ we have $\int_{R_i}\Delta P(w) dm(w) \leq 1\leq P(R_i).$ 
The conclusion that $\displaystyle\sum_{i=1}^N \int_{R_i}\Delta P(w) dm(w) \leq \delta$ proves the claim.
}\end{ex}

\section{Uniform Global Structures}\label{sec:uniformglobalstructures}

Before we prove Theorem \ref{ugschar}, we need a technical lemma. 

\begin{lem}\label{necessary} If ${\rm b}\Omega$ has a UGS, then there are constants $0<C_1<C_2$, depending on $\Delta P$ and $\delta_0$, such that
\begin{itemize}
	\item[(a)] $\displaystyle\phantom{p}\inf_{z_0\in\mathbb{C}}\ \sup_{z\in B(z_0,\delta)}\ \sup_{0<\hat\delta\leq \delta}\ (\hat\delta+\hat\delta^2)^{-1}\int_{B(z,\hat\delta)} \Delta P(w) dm(w)\geq C_1,\qquad \delta\geq \delta_0;$
	\item[(b)] $\displaystyle\sup_{z_0\in\mathbb{C}}\ \sup_{\delta>0}\ (\delta+\delta^2)^{-1}\int_{B(z_0,\delta)} \Delta P(w) dm(w)\leq C_2.$
\end{itemize}
\end{lem}

\begin{proof}
Because the proofs of both parts are similar, we only include the proof of (a). 
If (a) does not hold, then for any arbitrarily large $\delta$ and small $\epsilon>0$, we can find $z_0\in\mathbb{C}$ such that \[\sup_{z\in B(z_0,\delta)}\ \sup_{0<\hat\delta\leq \delta}\ (\hat\delta+\hat\delta^2)^{-1}\int_{B(z,\hat\delta)} \Delta P(w) dm(w)\leq \epsilon.\]
Let $R=(R_1,\ldots,R_N)$ be a $(z_0,\delta)$-stockyard. Because $R_i\subset B(z,P(R_i))$ for any $z\in R_i\subset B(z_0,\delta)$, we have
$$\int_{R_i} \Delta P(w) dm(w) \leq \int_{B(z,P(R_i))} \Delta P(w) dm(w) \leq \epsilon(P(R_i)+P(R_i)^2),$$
and therefore, because $\sum P(R_i)\leq \delta$,
$$\displaystyle\sum_{R_i\in R} \int_{R_i}\Delta P(w) dm(w) \leq \epsilon\displaystyle\sum_{R_i\in R} (P(R_i)+P(R_i)^2)\leq \epsilon(\delta+\delta^2).$$
Taking $\epsilon$ arbitrarily small, we see that there can be no UGS.
\end{proof}
\smallskip

We are now in a position to prove Theorem \ref{ugschar}.

\begin{proof}[Proof of Theorem \ref{ugschar}]
Fix a point $z_0\in\mathbb{C}$ and $\delta\geq \delta_0$. Choose a $(z_0,\delta)$-stockyard $R=(R_1,\ldots,R_N)$ with $\displaystyle\sum_{R_i\in R} \int_{R_i}\Delta P(w) dm(w) \approx \Lambda(p_0,\delta).$

For the bound $f(\delta)\lesssim \delta^2$ (when $\delta_0\geq 1$), Lemma \ref{necessary}(b) gives
\[\displaystyle\sum_{R_i\in R} \int_{R_i} \Delta P(w) dm(w) \leq \displaystyle\sum_{R_i\in R} \int_{B(z_i,P(R_i))} \Delta P(w) dm(w) \lesssim \delta^2.\]

For the lower bound, choose $\hat\delta\leq\delta$ and $z\in B(z_0,\delta)$ such that $$(\hat\delta+\hat\delta^2)^{-1}\int_{B(z,\hat\delta)}\Delta P(w) dm(w) \gtrsim 1.$$
Taking a $(z_0,3\hat\delta)$-stockyard $R$ which includes $N\approx\delta\hat\delta^{-1}$ pens of the form $B(z,\hat\delta)$, we have \[f(3\delta)\geq\displaystyle\sum_{R_i\in R} \int_{R_i} \Delta P(w) dm(w) \gtrsim \delta\hat\delta^{-1}(\hat\delta+\hat\delta^2)\gtrsim\delta,\]
or $f(\delta)\gtrsim \delta$ if $\delta\geq 3\delta_0$.
This concludes the proof.
\end{proof}
\smallskip

The next theorem shows that, if $\|\Delta P\|_\infty<\infty$, then there is only one possible UGS.

\begin{thm}\label{approxquad} If $\|\Delta P\|_\infty<\infty$ then the following are equivalent:
	\begin{itemize}
		\item[(a)] ${\rm b}\Omega$ has a UGS.
		\item[(b)] $(\delta^2,\delta_0)$ is a UGS for ${\rm b}\Omega$.
		\item[(c)] $\displaystyle |B(z,\delta)|^{-1}\int_{B(z,\delta)} \Delta P(w)dm(w)\approx 1,$ uniformly in $z\in\mathbb{C}$ and $\delta\geq \delta_0$. 
	\end{itemize}
\end{thm}

\begin{proof}
We first note that (b) trivially implies (a).

If (a) holds, then choose $0<C_1<C_2$ satisfying the conclusion of Lemma \ref{necessary}.
Noting that\newline $\delta^{-1}\int_{B(z,\delta)} \Delta P(w) dm(w) \leq \pi\delta\|\Delta P\|_\infty$, we obtain
$$\displaystyle\inf_{z_0\in\mathbb{C}} \sup_{|z_0-z|<\delta}\sup_{\tilde{C}_1\|\Delta P\|_\infty^{-1}<\hat\delta\leq \delta} (\hat\delta+\hat\delta^2)^{-1}\int_{B(z,\hat\delta)} \Delta P(w) dm(w)>C_1.$$
Taking $\delta=\delta_0$, we conclude that for every $z_0\in\mathbb{C}$, $B(z_0,\delta_0)$ contains a disc whose radius is comparable to $\delta_0$ and on which the average of $\Delta P$ is bounded away from zero. This in turn implies that
$$\inf_{z_0\in\mathbb{C}} \delta_0^{-2} \int_{B(z_0,\delta_0)} \Delta P(w) dm(w) \geq\tilde{c}>0.$$
By covering large balls with balls of radius $\delta_0$ and applying the Vitali covering lemma, we conclude the proof of (c).

That (c) implies (b) follows by taking a stockyard which consists of one large circular pen.
\end{proof}
\smallskip

\section{Estimates for $d$}\label{sec:metric}

In this section we obtain an approximate formula for $d$ if ${\rm b}\Omega$ has a UGS. To begin, we parametrize large CC balls with non-isotropic cylinders. 

\begin{thm}\label{largescaleparam}Suppose that ${\rm b}\Omega$ has uniform global structure $(f(\delta),\delta_0)$, and define $Cyl_d (p_0,\delta):={\rm Img}(\Psi_{p_0,\delta})$, where $\Psi_{p_0,\delta}:B(0,1)\times(-1,1)\rightarrow {\rm b}\Omega,$
\begin{equation*}
\Psi_{p_0,\delta}(a,b,c)=\exp(a\delta X+b\delta Y+cf(\delta)\partial_{t})(p_0).
\end{equation*}
If $\delta\geq \delta_0$, then there exists $R>r>0$ (independent of $\delta$ and $p_0$) such that \[Cyl_d (p_0,r\delta)\subset B_d(p_0,\delta)\subset Cyl_d (p_0,R\delta).\]
\end{thm}

\begin{proof}
Let $(\alpha,\beta)\in\mathscr{X}$ and $a=\int_0^1 \alpha(t)dt,\ b=\int_0^1 \beta(t)dt.$ 
To begin, write 
\[\Lambda_{z_0,\delta}(\alpha,\beta)-\Lambda_{z_0,\delta}(a,b)=cf(\delta)+\oint_{\pi\gamma-\pi\tilde{\gamma}} \partial_{y}Pdx-\partial_{x}Pdy,\]
where $\gamma$ and $\tilde{\gamma}$ are the curves determined by $(\alpha,\beta)$ and the constant functions $(a,b)$, respectively. 
By Theorem \ref{gschar}, $|\Lambda_{z_0,\delta}(\alpha,\beta)-\Lambda_{z_0,\delta}(a,b)|\leq Kf(\delta)$ for some universal $K>0$. 
With this bound in hand, we proceed.

Suppose that $(x,y,t)=(x_0+\frac{a\delta}{2},y_0-\frac{b\delta}{2},t_0+\Lambda_{z_0,\delta}(\alpha,\beta))$. 
We would like $(x,y,t)\in Cyl_d (p_0,R\delta)$ for some $R$ depending only on $K$. 

Consider the element of $Cyl_d (p_0,R\delta)$ given by 
\begin{equation*}
\Psi_{p_0,R\delta}\Big(\frac{a}{R},\frac{b}{R},c\Big) = \Big(x_0+\frac{a\delta}{2},y_0-\frac{b\delta}{2},t_0 + cf(R\delta)+\Lambda_{z_0,R\delta}\Big(\frac{a}{R},\frac{b}{R}\Big)\Big).
\end{equation*}
We want to prove that, if $R$ is large enough, then there exists $c\in (-1,1)$ such that 
\begin{equation} cf(R\delta)+\Lambda_{z_0,R\delta}\Big(\frac{a}{R},\frac{b}{R}\Big)=\Lambda_{z_0,\delta}(\alpha,\beta).\label{eq:UGS1}\end{equation}

But $\Lambda_{z_0,R\delta}\Big(\frac{a}{R},\frac{b}{R}\Big)=\Lambda_{z_0,\delta}(a,b)$, so $|\Lambda_{z_0,\delta}(\alpha,\beta)-\Lambda_{z_0,R\delta}\Big(\frac{a}{R},\frac{b}{R}\Big)|\leq Kf(\delta)$. By Theorem \ref{ugschar}, there exists $R$ such that $f(R\delta)\geq Kf(\delta)$, and therefore $c$ can be chosen to satisfy (\ref{eq:UGS1}). 
This proves the second containment.

For the first containment, it suffices to show that there is $R>0$ such that if $\delta\geq \delta_0$, then $Cyl_d (p_0,\delta)\subset B_d(p_0,R\delta).$ 

If $p=\Psi_{p_0,\delta}(a,b,c)\in Cyl_d(p_0,\delta)$, then 
\begin{equation*}
p=\Psi_{(z_0,t_0+cf(\delta)),\delta}(a,b,0)\in B_d((z_0,t_0+cf(\delta)),\delta).
\end{equation*}
We will show that there is a universal $R>0$ with $(z_0,t_0+cf(\delta))\in B_d(p_0,R\delta),$ which implies that $p\in B_d(p_0,(R+1)\delta)$.

Choose $(\alpha,\beta)\in\mathscr{X}^\ast$ such that $$\Lambda_{z_0,R\delta}(\alpha,\beta)\approx f(R\delta)\gtrsim Rf(\delta),$$ where the last inequality holds by Theorem \ref{ugschar}. 
Now choose $R$ so that \newline $\Lambda_{z_0,R\delta}(\alpha,\beta)\geq |cf(\delta)|$. 
By continuity there exists $1>\epsilon>-1$ such that $\Lambda_{z_0,R\delta}(\epsilon\alpha,\epsilon\beta)=cf(\delta)$, which completes the proof.
\end{proof}
\smallskip

\begin{defn}{\rm 
We say that ${\rm b}\Omega$ is \emph{uniformly of type $m$} if \[ \displaystyle \sup_{0\leq j\leq m-2}\  \sup_{j^{th}\ order\ derivatives\ W} |W\Delta P(z_0)| \approx 1.\]
}\end{defn}

The following is a consequence of the previous theorem and the results in \cite{NagelSteinWainger1985} or \cite{Street2011}.

\begin{cor}\label{globalparam}Suppose that ${\rm b}\Omega$ has a UGS and that $P$ is uniformly of type $m$. Then there exist uniform constants $0<r<R$ and diffeomorphisms $\Psi_{p_0,\delta}:B(0,1)\times(-1,1)\rightarrow\mathbb{R}^3$ with $\Psi_{p_0,\delta}(0)=p_0$ and 
\[{\rm Img}(\Psi_{p_0,r\delta})\subset B_d(p_0,\delta)\subset Img(\Psi_{p_0,R\delta}).\] 
Moreover, we have the following estimate for the Jacobian determinant of $\Psi_{p_0,\delta}${\rm:} $|J\Psi_{p_0,\delta}|\approx \delta^2\Lambda(p_0,\delta)$
\end{cor}
\smallskip

\begin{cor} If ${\rm b}\Omega$ has a UGS, then $|B_d(p_0,\delta)|\approx \delta^2 \Lambda(p_0,\delta).$\end{cor}
\smallskip

For $p_0=(z_0,t_0)$ and $p_1=(z_1,t_1)$, define 
\begin{eqnarray*}
T(p_0,p_1)&:=&\begin{cases} 2\Im\left( \int_0^1 (z_0-z_1)\partial_{z}P(z_0+(z_1-z_0)s)ds\right) ,& d(p_0,p_1)\geq \delta_0, \\
 2\Im\left(\sum_{k=1}^m \partial_{z}^kP(z_1) \frac{(z_0-z_1)^k}{k!}\right),& d(p_0,p_1)\leq \delta_0.\end{cases}
\end{eqnarray*}
 We now get the following extremely useful result.
\smallskip

\begin{thm}\label{globalmetric} Suppose that $P$ satisfies the assumptions of Corollary \ref{globalparam}, and fix $p_0,p_1\in {\rm b}\Omega$. 
Then \[d(p_0,p_1)\approx |z_1-z_0| + \mu(p_0,|t_1-t_0-T(p_0,p_1)|),\] where $\mu(p_0,\delta)$ is the inverse function to $\delta\mapsto \Lambda(p_0,\delta).$
If, in addition,\newline $\|\nabla^2 P\|_\infty<\infty,$ then we have $d(p_0,p_1)\approx |z_0-z_1|+\sqrt{|t_0-t_1|}$ for $\delta\geq \delta_0$.
\end{thm}

\begin{proof}
Let $\delta=(1+\epsilon)d(p_0,p_1)$, for some small positive $\epsilon$.
For $\delta\lesssim \delta_0$, the result can be found in \cite{NagelSteinWainger1985}. 
If $\delta\gg \delta_0$, we write \begin{equation}\label{eq:expmapest} p_1=\exp(a\delta X+b\delta Y + cf(\delta)\partial_{t})(p_0),\qquad (a,b,c)\in B(0,1)\times(-1,1).\end{equation} 
Because $\delta\approx d(p_0,p_1)$, we have $|((a+ib),c)|\approx 1$. 
From (\ref{eq:expmapest}) we obtain $a=2\delta^{-1}(x_1-x_0),\ b=2\delta^{-1}(y_0-y_1)$.~
Solving (\ref{eq:expmapest}) we obtain $t_1=t_0+cf(\delta)+T(p_0,p_1),$
and therefore $|t_1-t_0-T(p_0,p_1)|=|c|f(\delta),\ \ |z_0-z_1|\approx |a+ib|\delta.$

The bounds on $((a+ib),c)$ imply that $\delta\gtrsim |z_0-z_1|+\mu(p_0,|t_1-t_0-T(p_0,p_1)|).$
On the other hand, $|a+ib|+|c|\approx 1$, so that $$\delta\lesssim |z_0-z_1|\qquad\mbox{or}\qquad \delta\lesssim \mu(p_0,|t_1-t_0-T(p_0,p_1)|),$$ and therefore $d(z_0,z_1)\approx |z_0-z_1|+\mu(p_0,|t_1-t_0-T(p_0,p_1)|).$

For the final statement, we remark that if $\Phi:\mathbb{C}^2\rightarrow \mathbb{C}^2$ is a linear biholomorphism, $\tilde{\Omega}=\Phi(\Omega)$, and $\tilde{d}$ is the CC metric on ${\rm b}\tilde{\Omega}$ induced by the real and imaginary parts of the Cauchy-Riemann vector field, then $\tilde d(\Phi(p_0),\Phi(p_1))=d(p_0,p_1)$.
We may therefore assume that $P(p_0)=0$ and $|\nabla P(p_0)|=0$.

One can then use the mean value theorem on $T(p_0,p_1)$, together with the condition that $\|\nabla^2 P\|_\infty < \infty$ and Theorem \ref{approxquad}, to obtain the result.
\end{proof}
\smallskip

\begin{rem}{\rm The condition that $\|\nabla^2 P\|_\infty<\infty$ may seem rather severe, but it holds if, for instance, $\Omega$ is convex and $\|\Delta P\|_\infty<\infty$.}\end{rem}

\section{Conclusion}\label{sec:conclusion}
The results of this paper have trivial extensions to decoupled domains (see \cite{McNeal1991}). The application of the methods developed here to more general types of domains will be the subject of future investigations.

This paper represents a crucial first step in the study of \emph{approximately quadratic domains}, or uniformly finite type domains $\Omega=\left\{ \Im(z_2)>P(z_1)\right\}\subset \mathbb{C}^2$ which possess uniform global structures.
The results of this paper, particularly Theorems \ref{approxquad} and \ref{globalmetric}, are fundamental for the study of the more complicated analytical objects associated to these domains, which will be the focus of future work.

\bigskip

\noindent Aaron Peterson\\
Department of Mathematics\\
University of Wisconsin - Madison\\
480 Lincoln Dr.\\
Madison, WI 53706, USA\\
e-mail: apeterso@math.wisc.edu
\end{document}